\documentclass[12pt]{amsart}

\usepackage[utf8]{inputenc}
\usepackage[T1]{fontenc}
\usepackage{amsfonts}
\usepackage{amssymb}
\usepackage{amsmath}
\usepackage{dsfont}
\usepackage{color}
\usepackage{url}
\usepackage{hyperref}
\usepackage{graphicx}
\newtheorem{theorem}{Theorem}[section]
\newtheorem{lemma}[theorem]{Lemma}
\newtheorem{proposition}[theorem]{Proposition}
\newtheorem{corollary}[theorem]{Corollary}

\theoremstyle{definition}
\newtheorem{df}{Definition}
\newtheorem{example}[df]{Example}
\newtheorem{remark}[df]{Remark}

\newcommand{\N}{\mathbb N}

\newcommand{\R}{\mathbb R}

\newcommand{\g}[1]{\boldsymbol{#1}}

\title[On functions preserving products of semimetric spaces]{On functions preserving products of certain classes of semimetric spaces}
\date{December 2020}
\author{Mateusz Lichman}
\address{Institute of Mathematics, Lodz University of Technology}
\email{mateusz.lichman@wp.pl}

\author{Piotr Nowakowski}
\address{Institute of Mathematics, Lodz University of Technology}
\email{piotr.nowakowski@dokt.p.lodz.pl}

\author{Filip Turobo\'s}
\address{Institute of Mathematics, Lodz University of Technology}
\email{filip.turobos@gmail.com}
\begin{document}

\begin{abstract}
In the paper we continue the research of Bors\'{i}k and Dobo\v{s} on functions which allow us to introduce a metric to the product of metric spaces. In this paper we extend their scope on broader class of spaces which usually fail to satisfy the triangle inequality, albeit they tend to satisfy some weaker form of this axiom. In particular, we examine the behavior of functions preserving $b$-metric inequality. We provide analogues of the results of Bors\'{i}k and Dobo\v{s}, adjusted to the new, broader setting. The results we obtained are illustrated with multitude of examples. Furthermore, the connections of newly introduced families of functions with the ones already known from the literature are investigated. 
\end{abstract}

\maketitle

\section{Introduction}

In the case of various structures of either algebraic type it is relatively easy to merge a finite number of such spaces into a single product space equipped with structure based on the initial ones. When metric structures are combined, usually functions like minimum, sum or quadratic mean are employed for this task. The investigation of functions which coalesce multiple metric spaces into a single one dates back to the beginning of 1980s', when Bors\'{i}k and Dobo\v{s} \cite{Borsik1981,Borsik1982} laid formal foundations to the concept of the product-wise metric preserving functions. The research on this concept was conducted in two different directions. Some researchers either sought more properties of such functions (e.g. \cite{Pokorny1996,Pokorny1998,Terpe1988}). Others considered a variation of this notion, motivated by some applications, where a family of metrics were defined on a same set and it was necessary to meld them into a single one. This approach resulted in the definition of aggregation functions -- compare with the neat applicational papers of Valero et al., \cite{Martin2011,Massanet2012,Mayor2010,Mayor2019} and see also the references therein.

The simultaneous development of the property preserving functions of one variable is also worth mentioning. The idea can be traced back to Wilson \cite{Wilson1935} and Sreenivasan \cite{Sreenivasan1947}. The research on this subject was later conducted by Bors\'{i}k, Dobo\v{s} and Piotrowski (see \cite{Borsik1981compo,Borsik1981,Borsik1982,Borsik1988,Dobos1994,Dobos1996,Dobos1996Piotro,Dobos1997,Dobos1998}), Corazza (see \cite{Corazza1999}), Das (see \cite{Das1989}), Dovgoshey (see \cite{Dovgoshey2019,Dovgoshey2020,Dovgoshey2009,Vallin2020}), J\r{u}za (see \cite{Juza1956}), Khemaratchatakumthorn, Pongsriiam, Termwuttipong and Samphavat (see \cite{Khemaratchakumthorn2018,Khemaratchakumthorn2019,Khemaratchakumthorn2012,Pongsriiam2014FPT,Pongsriiam2014,Samphavat2020,Termwuttipong2005}), Pokorny (see \cite{Pokorny1993,Pokorny1996,Pokorny1998}), Vallin (see \cite{Vallin1998,Vallin2000,Vallin2020}) and recently also by Jachymski and Turobo\'s (see \cite{Jachymski2020,arxiv}).

The concept of $b$-metric spaces can be traced back to Frink (see \cite{Frink1937}) as well as Bakhtin (see \cite{Bakhtin1989}), but the name is usually connected with Czerwik \cite{Czerwik1993}, who conducted the research on particular subclass of this spaces where the relaxing constant was 2. Recently this generalization of the notion of metric space is reliving its second youth, as the interest in this direction of metric-related research has been rekindled in recent 30 years. This renaissance refers both to the advances in the field of metric fixed point theory as well as properties of these spaces themselves (see \cite{Bessenyei2017,Jachymski1995,Karapinar2018,VanAn2014}). In the context of preserving triangle-like conditions, a great deal of work has been done by aforementioned Thai mathematicians as well as Dovgoshey.


\section{Preliminaries}
We begin with introducing several important notions as well as some examples depicting those. Throughout the paper, by $\R_+$ we will denote the set of non-negative reals, i.e., $\R_+ = [0,\infty).$ Let us start with the notion of semimetric.

\begin{df}
Let $X$ be a non-empty set and $d\colon X^2\to \R_+$. We say that $d$ is a \textit{semimetric} on $X$ if the following conditions are satisfied:
\begin{itemize}
    \item[(S1)] $\forall_{x,y\in X} \, d(x,y)=0 \iff x=y$;
    \item[(S2)] $\forall_{x,y\in X} \, d(x,y)=d(y,x)$.
\end{itemize}
\end{df}

Since the paper revolves around various generalizations of the triangle inequality, we start with the following two general definitions inspired by the definition of triangle function introduced by P\'ales and Bessenyei \cite{Bessenyei2017}:

\begin{df}
A function $f \colon \R_+^2 \to \R_+$ is called \textit{nonreducing} if, for any $a,b \in \R_+^2$, $f(a,b) \geqslant \max\{a,b\}$.
\end{df}
\begin{df}
Let $g \colon \R_+^2 \to \R_+$ be a nonreducing function. 
Let $(X,d)$ be a semimetric space satisfying the following condition in place of the triangle condition:
\begin{itemize}
\item[(G)] $\forall_{x,y,z \in X_i} \, d(x,z) \leqslant g(d(x,y),d(y,z)).$
\end{itemize}
Then we say that $(X,d)$ is a \textit{$G$-metric space} (and, accordingly, that $d$ is a \textit{$G$-metric}), where condition $(G)$ is called a \textit{semi-triangle condition generated by the function $g$}.
\end{df}

\begin{example}\label{example:3} We now give examples of some popular conditions with function $g \colon \R_+^2 \to \R_+$ generating them (here $K\geqslant 1$ is some fixed, arbitrary constant):
\begin{enumerate} 
\item[(M)] $\forall_{x,y,z \in X_i} \,\, d(x,z) \leqslant d(x,y)+d(y,z),\,\,g(a,b) = a + b,$
\item[(U)] $\forall_{x,y,z \in X_i}\,\, d(x,z) \leqslant \max\{d(x,y),d(y,z)\},\,\, g(a,b) = \max\{a, b\},$
\item[($B_K$)] $\forall_{x,y,z \in X_i}\,\, d(x,z) \leqslant K(d(x,y)+d(y,z)),\,\, g(a,b) = K(a+ b).$
\item[($S_K$)] $\forall_{x,y,z \in X_i} \,\,d(x,z) \leqslant Kd(x,y)+d(y,z),\,\, g(a,b) = Ka+ b,$
\item[(T)] $\forall_{x,y,z \in X_i}\,\, d(x,z) \leqslant \psi(d(x,y)+d(y,z)), \,\,g(a,b) = \psi(a+ b),$ \text{ where } $\psi \colon \R_+ \to \R_+ \text{ is nondecresing,}$ continuous at the origin and  $\psi(0)=0.$
\end{enumerate}
While the first two examples depict well-known notions of \textbf{metric} and \textbf{ultrametric} respectively, third and fourth refer to the concepts known as $b$\textbf{-metric} and \textbf{strong} $b$\textbf{-metric spaces}. We refer the Reader to the appropriate chapters of \cite{KirkShahzad} for more examples and application of these concepts and we underline the fact that $K$ is fixed in this situations, which will be important later on. Last point refers to the notion of \textbf{triangle function} introduced in \cite{Bessenyei2017}.
\end{example}


Let $n \in \N$. For $i \in \{1,\dots,n\}$, let $g_i\colon \R_+^2 \to \R_+$ be a nonreducing function and $G_i$ be the semi-triangle condition generated by the function $g_i$. Let $(X_i,d_i)$ be $G_i$-metric space for $i=1,\dots, n$. Moreover, let $h\colon \R_+^2 \to \R_+$ be a nonreducing function and $H$, analogously, be the condition generated by the function $h$.

By $\g{X} = \prod_{i=1}^n X_i$ we shall denote the Carthesian product of spaces $X_i$. We want to describe all functions $F\colon \R^n_+ \to \R_+$ for which the function $D\colon \g{X}^2 \to \R_+$ given by a formula
\begin{equation}\label{functio}
D(\g{x},\g{y}) = F(d_1(x_1,y_1),\dots,d_n(x_n,y_n)),
\end{equation}
    for any\footnote{In general, throughout the paper the elements of the Carthesian product will be denoted by boldface font.} $\g{x}=(x_1,\dots,x_n)$, $\g{y}=(y_1,\dots,y_n) \in \g{X}$ is an $H$-metric on $\g{X}$. Thus, we arrive at the following definition.

\begin{df}
We say that $F$ is $(G_1,\dots,G_n)-H$\textit{-preserving} if for any collection of $G_i$-metric spaces $(X_i,d_i)$, $i=1,\dots,n$, the function $D$ given by \eqref{functio} is an $H$-metric on $\g{X}$. The family of all $(G_1,\dots,G_n)-H$-preserving functions will be denoted by $P_{(G_1,\dots,G_n)-H}$.
\end{df}

\begin{remark}
This definition is coherent with the notation used in \cite{Jachymski2020} and references therein. For example, the class of all functions which combine a finite family of $b$-metrics (regardless of the relaxation constants on each space) into a single $b$-metric on a product space will  be called ($n$)-$b$-metric preserving functions and the family of all such functions will be denoted by $P^n_{B}$. When different axioms are considered, for example functions combining multiple $b$-metrics into a metric, a name ($n$)-$b$-metric-metric preserving functions will be used. Again, the class of functions having this property will be denoted by $P^n_{BM}$.  Analogously, we define families $P^n_{SB}$, $P^n_{SM}$ etc.

Of course, when we want to specify the relaxation constants $K$ in starting or resulting spaces, we will refer to the notation $(B_K), (S_K)$ and so forth.
\end{remark}


Let us begin with the well-known example in order to acquaint the Reader with the motivation for this paper. 

\begin{example}
Consider a standard Euclidean metric $d_e$ on $\R^n$. It is, in fact, result of combining (according to formula \eqref{functio}) a function $F_1:R_+^n\to \R$ given by
\[
F_1(x_1,\dots,x_n):= \sqrt{ \sum_{i=1}^n x_i^2 }
\]
with $n$ copies of standard metric $(x,y)\mapsto |x-y|$ defined on $\R$.

Replacing $F_1$ with $F_2:\R_+^n\to \R_+$ defined as
\[
F_2(x_1,\dots,x_n):= \sum_{i=1}^n x_i
\]
yields so called taxicab metric.\footnote{Also known as Manhattan metric or simply $L_1$ metric.}  In turn using function $F_3$ defined as maximum of its arguments yields metric known as Chebyshev distance or $L_\infty$ metric. 
\end{example}

The following lemma is an easy consequence of the definition of families $P_{(G_1,\dots,G_n)-H}$.
\begin{lemma} \label{zal}
Let $n \in \N$. For $i \in \{1,\dots,n\}$, let $g_i,h_i,g,h\colon \R_+^2 \to \R_+$ be nonreducing functions and $G_i$, $H_i,G,H$ be the semi-triangle conditions generated, respectively, by the functions $g_i, h_i,g,h$. Assume that $F\colon \R^n_+ \to \R_+$ is $(G_1,\dots,G_n)-G$-preserving. Then:
\begin{itemize}
    \item[(i)] if for all $i \in \{1,\dots,n\} $ and every $a,b \in \R_+$ we have $g_i(a,b) \geqslant h_i(a,b)$ (equivalently $H_i \Rightarrow G_i$), then $F$ is $(H_1,\dots,H_n)-G$-preserving;
    \item[(ii)] if for all $a,b \in \R_+$ we have $g(a,b) \leqslant h(a,b)$ (equivalently $G \Rightarrow H$), then $F$ is $(G_1,\dots,G_n)-H$-preserving.
\end{itemize}
\end{lemma}
\begin{remark}\label{uwaga}
The family of all ($n$)-$b$-metric-metric preserving functions can be written as the following intersection.
\[
P_{BM}^n = \bigcap_{K_1,\dots,K_n \geqslant 1} P_{(B_{K_1},B_{K_2},\dots,B_{K_n})-M}.
\]
However, if two \textit{relaxed} conditions are concerned, this class becomes an intersection of unions of particular families of functions (this fact is an immediate consequence of Lemma \ref{mainlemma}, which will be proven in the latter part of the paper). For example, when the class $P_{B}^n$ of all ($n$)-$b$-metric preserving functions is considered, the following equality holds 
\[
P_{B}^n = \bigcap_{K_1,\dots,K_i \geqslant 1} \bigcup_{K\geqslant 1} P_{(B_{K_1},B_{K_2},\dots,B_{K_n})-B_{K}}.
\]

Of course, analogous statements can be given for families connected with inequalities (S) and so on.
\end{remark}

As one can see, the functions which are of the main concern in this paper have multiple arguments. This require us to specify several notions, which are less obvious than their one-dimensional equivalents. 

\begin{df}
Let $n \in \N$. We define a relation $\preceq$ on $\R_+^n$ by: $\g{a} \preceq \g{b}$ iff $a_i\leqslant b_i$ for all $i\in \{1,\dots , n\}$ where $\g{a}=(a_1, \dots, a_n)$ and $\g{b}=(b_1, \dots, b_n)$. We will also write
$\g{a+b} := (a_1+b_1,\dots,a_n+b_n)$ and $\g{0} :=(0,\dots,0)$.
\end{df}

Obviously, $\preceq$ is a partial order. The same partial order was considered in the multiple papers of Valero et al. \cite{Martin2011,Massanet2012} as well as in the original works of Bors\'ik and Dobo\v{s}. In the light of the foregoing, we introduce the definitions of the following properties:

\begin{df}
Let $n \in \N$. A function $F\colon\R_+^n\to\R_+$ is said to be \begin{itemize}
    \item \textit{monotone}, if $F(\g{a})\leqslant F(\g{b})$ whenever $\g{a} \preceq \g{b}$;
    \item \textit{subadditive}, if $F(\g{a+b})\leqslant F(\g{a})+F(\g{b})$ for each $\g{a}, \g{b} \in \R_+^n$;
    \item \textit{quasi-subadditive}, if there exists $s\geqslant 1$ such that $F(\g{a+b})\leqslant s(F(\g{a})+F(\g{b}))$ for each $\g{a}, \g{b} \in \R_+^n$.
\end{itemize}
\end{df}

Another crucial property of one-dimensional metric-preserving functions is known under the name \textit{amenability}. A function $f\colon \R_+\to\R_+$ is called \textit{amenable} if $f^{-1}\left(\{0\}\right) = \{0\}$. It seems reasonable to define its multidimensional equivalent as follows:

\begin{df}
A function $F\colon \R_+^n\to \R_+$ is said to be \textit{amenable} if the following equivalence holds: \[
F(\g{x}) = 0 \iff 
\g{x}=(0,0,\dots,0).\]
\end{df}

\section{Main results}

We'd like to start with a characterization of the $(G_1,\dots,G_n)-H$ preserving functions which is based on the concept of triangle triplet. The notion of the triangle triplet can be traced back to the paper of Sreenivasan \cite{Sreenivasan1947} and was reintroduced in \cite{Borsik1981compo}. It is worth pointing out, that the usage of this concept for more general spaces appears in \cite{Khemaratchakumthorn2018,Pongsriiam2014} as well as in \cite{arxiv}. A successful attempt of extending the original definition into multidimensional version have also been undertaken, see \cite{Borsik1981}. 

We will now unify all the concepts discussed above via the following definition:

\begin{df}\label{onedimtt}
Let $G$ be the semi-triangle condition generated by the nonreducing function $g\colon \R_+^2 \to \R_+$. A triplet $(a,b,c) \in [0,+\infty)^3$ is called a $G$\textit{-triangle triplet} if  $a \leqslant g(b,c), b \leqslant g(c,a), c \leqslant g(a,b)$ and $a,b,c > 0$ or 
 $(a,b,c)=(0,l,l)$ for some non-negative constant $l$, up to a permutation. 
\end{df}
\begin{remark}
This specification in the definition above enables us to avoid the pitfall of triplets consisting of $0$ and two distinct positive numbers. Such triplet of distances could satisfy the desired inequalities from the Definition \ref{onedimtt} despite being unobtainable in any semimetric space due to the symmetry axiom. These triplets could cause additional issues in formulation of some Theorems as well as in the proofs, therefore we have decided to exclude them from the Definition above.
\end{remark}
We will use the symbol $\triangle_G$ to denote the set of all $G$-triangle triplets (similar to the notions used in \cite{Samphavat2020} as well as in their previous papers).

\begin{df}\label{multidimtt}
Let $n\in \N$. For all $i \in \{1,\dots,n\}$ let $G_i$ be the semi-triangle condition generated by the nonreducing function $g_i\colon \R_+^2 \to \R_+$. A triplet $(\g{a},\g{b},\g{c}) \in [0,+\infty)^{3n}$, where $\g{a},\g{b},\g{c} \in \R_+^n$ is called a $(G_1,\dots,G_n)$\textit{-triangle triplet} if for each $i=1,\dots,n$, the $i$-th coordinates of $\g{a},\g{b}$ and $\g{c}$ form a $G_i$-triangle triplet (according to the Definition \ref{onedimtt}). 
\end{df}

Analogously to the one-dimensional case, we will denote the set of all $(G_1,\dots,G_n)$-triangle triplets by $\triangle_{(G_1,\dots,G_n)}$. 

We can now move on to extending the results obtained in \cite{Borsik1981,Corazza1999,Dobos1998,Samphavat2020,arxiv} to a multidimensional scope.

\begin{theorem}\label{theorem:main}
Let $n \in \N$. For $i \in \{1,\dots,n\}$, let $g_i\colon \R_+^2 \to \R_+$ be a nonreducing function,  $G_i$ be the semi-triangle condition generated by the function $g_i$. Moreover, let $h\colon \R_+^2 \to \R_+$ be a nonreducing function and $H$ be the semi-triangle condition generated by the function $h$.
A function $F\colon \R^n_+ \to \R_+$ is $(G_1,\dots,G_n)-H$-preserving if and only if it is amenable and satisfies the following condition
\begin{equation} \label{w}
\forall_{(\g{a},\g{b},\g{c})  \in \triangle_{(G_1,\dots,G_n)}} \ \    (F(\g{a}),F(\g{b}),F(\g{c}))\in \triangle_{H}.
\end{equation}
\end{theorem}

\begin{proof}
"$\Leftarrow$ "
For $i\in \{1,\dots, n\}$, let $(X_i,d_i)$ be a $G_i$-metric space. Without loss of generality, we may assume that these are non-trivial, i.e. consists of at least two points.
Let $\g{X} = \prod_{i=1}^n X_i$.
Let $\g{x}=(x_1,\dots,x_n),\g{y} =(y_1,\dots,y_n), \g{z} =(z_1,\dots,z_n)\in \g{X}$. Since $F$ is amenable, we have 
\begin{eqnarray*}
D(\g{x},\g{y}) =0 &\Leftrightarrow& F(d_1(x_1,y_1),\dots,d_n(x_n,y_n))=0\\
&\Leftrightarrow& \forall_{i\in\{1,\dots,n\}}\  d_i(x_i,y_i)=0\\ &\Leftrightarrow& \g{x}=\g{y}.
\end{eqnarray*}
By the definition of $D$ and the fact that $d_i$ are $G_i$-metrics, we have that $D$ is symmetric.

Since $d_i$ are $G_i$-metrics, we have that, for any $i \in \{1,\dots,n\}$, 
$$d_i(x_i,z_i) \leqslant g_i(d_i(x_i,y_i),d_i(y_i,z_i)),$$
$$d_i(x_i,y_i) \leqslant g_i (d_i(y_i,z_i),d_i(x_i,z_i)),$$
$$d_i(y_i,z_i) \leqslant g_i (d_i(x_i,z_i),d_i(x_i,y_i)),$$
so $(d_i(x_i,z_i),d_i(x_i,y_i),d_i(y_i,z_i))$ \text{forms a }$G_i$\text{-triangle triplet}.
Hence, by (\ref{w}), we have that
\begin{eqnarray*}
D(\g{x},\g{z}) &=& F(d_1(x_1,z_1),\dots,d_n(x_n,z_n))\\
&\leqslant &  h(F(d_1(x_1,y_1),\dots,d_n(x_n,y_n)), F(d_1(y_1,z_1),\dots,d_n(y_n,z_n)))\\
&=& h(D(\g{x},\g{y}),D(\g{y},\g{z})).
\end{eqnarray*}
Therefore, $F$ is $(G_1,\dots,G_n)-H$-metric preserving.

"$\Rightarrow $"

Let $F$ be $(G_1,\dots,G_n)-H$-metric preserving. On the contrary, assume that $F$ is not amenable. If $F(\g{0}) \neq 0$, then a contradiction is obvious. 

Thus, there is $\g{a}=(a_1,\dots,a_n) \in \R^n_+ \setminus \{\g{0}\}$ such that $F(\g{a}) = 0$. For $i \in \{1,\dots,n\}$ let $(X_i,d_i) = (\{0,a_i\},d_e)$, where $d_e$ is the Euclidean metric. Then, for any $i \in \{1,\dots, n\}$, $(X_i,d_i)$ is, obviously, a $G_i$-metric space. Indeed, the first two axioms are trivial and the third follows from the fact that function $g_i$ is nonreducing for each $i$. 

Let $x_i=0, y_i = a_i$, for $i \in \{1,\dots, n\}$, $\g{x}=(x_1,\dots,x_n),\g{y}=(y_1,\dots,y_n).$ Then $\g{x} \neq \g{y}$ and $$D(\g{x},\g{y}) = F(d_1(x_1,y_1),\dots,d_n(x_n,y_n))= F(a_1, \dots, a_n) = 0,$$
therefore we arrive at the contradiction.

Now, assume on the contrary that condtion (\ref{w}) is not satisfied. Then, there exist $\g{a}=(a_1,\dots,a_n), \g{b} =(b_1,\dots,b_n), \g{c} =(c_1,\dots,c_n) \in \R_+^n$ such that, for any $i \in \{1,\dots, n\}$, $(a_i,b_i,c_i)$ is a $G_i$-triangle triplet, but $F(\g{c}) > h(F(\g{a}),F(\g{b}))$. For $i \in \{1,\dots,n\}$ define the space $(X_i,d_i)$ in the following way. Let $X_i = \{x_i,y_i,z_i\}$ (where $x_i,y_i,z_i$ are arbitrary but not necessarily distinct points -- in the case where some of the values below equal $0$) and 
\begin{eqnarray*}
d_i(x_i,y_i) &= d_i(y_i,x_i) &= a_i,\\
d_i(y_i,z_i) &= d_i(z_i,y_i) &= b_i,\\ 
d_i(x_i,z_i) &= d_i(z_i,x_i) &= c_i,\\
d_i(x_i,x_i) &= d_i(y_i,y_i) &= d_i(z_i,z_i) = 0.
\end{eqnarray*}
Then, for any $i \in \{1,\dots, n\}$, $(X_i,d_i)$ is $G_i$-metric space. Indeed, the first two axioms hold trivially. The third one follows from the fact that $g_i$ are nonreducing and the assumpion that $(a_i,b_i,c_i)$ is a $G_i$-triangle triplet. Let $\g{x}=(x_1,\dots,x_n),\g{y}=(y_1,\dots,y_n),\g{z}=(z_1,\dots,z_n).$ We have 
\begin{eqnarray*}
D(\g{x},\g{z}) &=&F(d_1(x_1,z_1),\dots,d_n(x_n,z_n))\\
&=&F(\g{c})> h(F(\g{a}),F(\g{b}))\\
&=& h(F(d_1(x_1,y_1),\dots,d_n(x_n,y_n)),F(d_1(y_1,z_1),\dots,d_n(y_n,z_n)))
\\ &=&h(D(\g{x},\g{y}),D(\g{y},\g{z})).
\end{eqnarray*}
Thus, we arrive at another contradiction, which finishes the proof.
\end{proof}
\begin{example} \label{srednia aryt}
Let $n\in \N, F\colon \R^n_+ \to \R_+$ be given by the formula $$F(a_1,\dots,a_n)=\frac{a_1+\dots+a_n}{n}$$ and let $K \geqslant 1$. We will show that $F$ is $(B_K,\dots,B_K)-B_K$ metric preserving. $F$ is obviously amenable. We need to prove that it satisfies condition (\ref{w}). For $i \in \{1,\dots,n\}$ let $a_i, b_i,c_i \in \R_+$ be such that $(a_i,b_i,c_i) \in \triangle_{B_K}$ for $i =1,2,\dots,n.$ Then we have
\begin{eqnarray*}
F(c_1,\dots,c_n) &=& \frac{c_1+\dots+c_n}{n}\\
&\leqslant& \frac{K(a_1+b_1)+\dots+K(a_n+b_n)}{n}\\
&=&K\cdot \frac{a_1+\dots+a_n}{n}+K\cdot \frac{b_1+\dots+b_n}{n} \\&=& K(F(a_1,\dots,a_n)+F(b_1,\dots,b_n)).
\end{eqnarray*}
Similarly, we prove that 
$$F(a_1,\dots,a_n) \leqslant K(F(c_1,\dots,c_n)+F(b_1,\dots,b_n))$$
and
$$F(b_1,\dots,b_n) \leqslant K(F(c_1,\dots,c_n)+F(a_1,\dots,a_n)).$$
Therefore, $(F(a_1,\dots,a_n),F(b_1,\dots,b_n),F(c_1,\dots,c_n)) \in \triangle_{B_K}$.
By Theorem \ref{theorem:main}, $F$ is $(B_K,\dots,B_K)-B_K$ metric preserving. In particular (for $K=1$), $F$ is $(n)$-metric preserving.

\end{example}
While arithmetic mean well preserves the additional properties of semimetric spaces, this is not the case when geometric mean is considered.
\begin{example}
Let $F\colon \R^n_+ \to \R_+$ be given by the formula $$F(a_1, a_2, \dots,a_n)=\sqrt[n]{a_1 a_2 \dots a_n}.$$ Clearly, $F$ is not amenable for $n\geqslant 2$. Due to Theorem \ref{theorem:main}, it cannot preserve any semimetric properties.
\end{example}
The insightful Reader should notice already that if conditions $G_i$'s and $H$ are the $b$-metric or strong $b$-metric inequalities, then the aforementioned theorem provides the characterization only for fixed values of relaxation constants $K$ in the respective definitions of those inequalities. We shall now try to extend this characterization to all values of $K$. Let us begin with the following remark.

\begin{remark}\label{useful:remark}
For semimetric space $(X,d)$  satisfying either $(B_K)$ or $(S_K)$ with relaxation constant $K\geqslant 1$ it is worth pointing out that it satisfies the same condition for any $K^\prime \geqslant K$. Indeed, for any three points $x,y,z\in X$ we have (reasoning for condition $(S_K)$ is almost unchanged):
\[
d(x,z)\leqslant K\left( d(x,y)+d(y,z)\right) \leqslant K^\prime \left( d(x,y)+d(y,z)\right).
\] 
We shall use this observation later on.
\end{remark}

For subsequent results we will need two Lemmas of Turobo\'{s} \cite[Lemma 3.1, 3.2]{arxiv} which allow us to combine multiple $b$-metric spaces into a single one. These results are generalizations of well-known results from the theory of metric spaces.

\begin{lemma}\label{lemma3.1}
Let $(X_1, d_1)$, $(X_2, d_2)$ be a pair of disjoint, [strong] $b$-metric spaces  with relaxation constants $K_1, K_2 \geqslant 1$. Additionally, assume that both spaces have finite diameter, i.e., \[r_1 :=
\operatorname{diam}_{d_1} (X_1) = \sup_{x,y\in X_1} d_1(x,y) < \infty \]
\[r_2 :=
\operatorname{diam}_{d_2} (X_2) = \sup_{x,y\in X_2} d_2(x,y) < \infty \]
and $\g{X}:=X_1 \cup X_2$ has at least three elements. Then, there exists a [strong] $b$-metric $d : \g{X} \times \g{X} \to [0, +\infty)$ with relaxation constant $K := \max\{K_1, K_2\}$, $\operatorname{diam}(\g{X}) = \max\{r_1, r_2\}$ and satisfying
\[
d_{|X_1\times X_1} = d_1 \; \; \land \; \; d_{|X_2\times X_2} = d_2.
\]
\end{lemma}

\begin{lemma}\label{lemma3.2}
Let $(X_n,d_n)_{n\in\N}$ be an increasing family of [strong] $b$-metric spaces with a common, fixed relaxation constant $K > 1$  i.e. for every pair of indices $k_1\leqslant k_2$ we have $X_{k_1} \subseteq X_{k_2}$ and $d_{k_1} \subseteq d_{k_2}$ (which means $d_{k_1}(x, y) = d_{k_2}(x, y)$ for all $x, y \in X_{k_1}$). Then a pair $(\g{X}, d)$, where
\[
\g{X}:=\bigcup_{n\in\N} X_n, \qquad \qquad d(x,y):= d_{k_{\min} }(x, y)
\] 
and $k_{\min}$ is any index such that $x, y \in X_{k_{\min}}$
is a [strong] $b$-metric space with the relaxation constant $K \geqslant 1$.
\end{lemma}

Now, we are able to proceed with the following four-in-one lemma, describing the behaviour of $(n)$-$b$-metric preserving functions as well as those which are $(n)$-strong-$b$-metric-$b$-metric, $(n)$-strong-$b$-metric or
$(n)$-$b$-metric-strong-$b$-metric preserving. As the proof in all cases follows the same train of thought, we present only one of its variants.

\begin{lemma}\label{mainlemma}
Let $n\in\N$. Assume that $F$ is a function such that for every $n$-element collection of [strong] $b$-metric spaces $(X_i,d_i)$ with fixed relaxation constants $K_{i}$, $i=1,\dots,n$, the product space $(\g{X},D )$ -- where $\g{X}:=\prod_{i=1}^n X_i$ and $D$ is given by \eqref{functio} -- is a [strong] $b$-metric space. Then, there exists $K^\prime \geqslant 1$ such that:
\begin{itemize}
\item[a)] for every [strong] $(B_{K_{1}},\dots,B_{K_{n}})$-triangle triplets $\g{a},\g{b},\g{c}$, the values $f(\g{a})$, $f(\g{b})$ and $f(\g{c})$ form a [strong] $B_{K^\prime}$-triangle triplet.
\item[b)] the relaxation constant of the resulting space $(\g{X},D)$ is bounded by $K^\prime$.
\end{itemize}
\end{lemma}

\begin{proof}
Due to the fact, that each version of this lemma is proved by the exact same method, we will simply prove the version concerning $b$-metric spaces.

Let $f$ be a function satisfying the assumptions of our Lemma. For the sake of convenience, let us denote $\g{K}:=(K_{1},\dots,K_{n})$. 

Let us suppose that the contrary to a) holds. In particular, it implies that for every $k\in\mathbb{N}$ there exists a $(B_{K_1},\dots,B_{K_n})$-triangle triplet\footnote{Throughout the paper for denoting vectors coordinates we use subscript indexation. Thus, the sequence of vectors $\left(\g{a}^{(n)}\right)_{n\in\N}$ will be marked by their upper indexation. For example, $4$-th coordinate of $\g{a}^{(6)}$ will be denoted by $a_4^{(6)}$.} $(\g{a}^{(k)},\g{b}^{(k)},\g{c}^{(k)})$ for which $\left(f(\g{a}^{(k)}),f(\g{b}^{(k)}),f(\g{c}^{(k)})\right)$ fails to be a one-dimensional $B_k$-triangle triplet. 

Without loss of generality, we will assume that
\begin{equation}\label{actuallyfork}
\forall_{k\in\N} \qquad 
F(\g{a}^{(k)})>k\cdot\left( F(\g{b}^{(k)})+F(\g{c}^{(k)})\right),
\end{equation}
but this additional assumption on the order of elements appearing in the discussed inequality serves only increasing the clarity of the proof.
	
	Fix $j\in\{1,\dots,n\}$.
	Let $\hat{X}^{(1)}_j = X^{(1)}_j:=\{(1,j,1),(1,j,2),(1,j,3)\}$ and define $b$-metric $d_j^{(1)}$ on this set as follows: 
	\begin{eqnarray*}
	d_j^{(1)}((1,j,1),(1,j,2))&=&a_j^{(1)} \\ d_j^{(1)}((1,j,2),(1,j,3))&=&b_j^{(1)},  \\ d_j^{(1)}((1,j,1),(1,j,3))&=&c_j^{(1)}. 
	\end{eqnarray*} 
	
	Remaining values of function $d_j^{(1)}$ are the result of semimetric axioms. An easy observation is that when $a^{(1)}_j, b^{(1)}_j, c^{(1)}_j$ are non-zero, then $d_j^{(1)}$ is a $b$-metric with relaxation constant $K_j$ (this fact is an immediate consequence of the Definition \ref{multidimtt}). 
	
	However, if one or more of these values (for example $a_j^{(1)}$) equals zero, then we can consider spaces consisting of less than three points simply by glueing some of those points together. As the Reader will see in the latter part of the proof, this identification of two or three points as a single one for these situations holds for the rest of the reasoning. Therefore, we will not focus on these special cases and for the sake of clarity we will assume that $a_j^{(k)},b_j^{(k)},c_j^{(k)}>0$.
	
    Now, assume that we have defined $\left(\hat{X}_j^{(k-1)},\hat{d}_j^{(k-1)}\right)$, which is a finite (thus bounded) $b$-metric space with relaxation constant $K_j$. We now proceed with defining $\hat{X}^{(k)}$. Put
	\(
    X^{(k)}_j:=\{(k,j,1),(k,j,2),(k,j,3)\} \) and  \(
    d^{(k)}_j:X_j^{(k)}\times X_j^{(k)} \to [0,+\infty),
    \)
    as a unique semimetric satisfying 
    \begin{eqnarray*}
    {d^{(k)}_j}((k,j,1),(k,j,2)) &=& a^{(k)}_j, \\
    {d^{(k)}_j}((k,j,2),(k,j,3)) &=& b^{(k)}_j, \\
    {d^{(k)}_j}((k,j,1),(k,j,3)) &=& c^{(k)}_j.
    \end{eqnarray*}
    The space $(X^{(k)}_j,d^{(k)}_j)$ is a bounded $b$-metric space with relaxation constant $K_j$. We can now use Lemma \ref{lemma3.1}, to obtain $b$-metric space with same relaxation constant, defined on a set $\hat{X}^{(k)}_j:=X^{(k)}_j\cup \hat{X}^{(k-1)}_j$. The Lemma guarantees that the extension $\hat{d}^{(k)}_j$ equals $d^{(k)}_j$ on $X^{(k)}_j\times X^{(k)}_j$ and $\hat{d}^{(k-1)}_j$ on $\hat{X}^{(k-1)}_j\times\hat{X}^{(k-1)}_j$. 
    
    By this step-wise procedure, we obtain a sequence of $b$-metric spaces $\left(\left(\hat{X}^{(k)}_j,\hat{d}^{(k)}_j\right) \right)_{n\in\mathbb{N}}$. Due to Lemma \ref{lemma3.2}, we thus obtain a $b$-metric space $(X_j,D_j)$ where $    X_j:=\bigcup_{k\in\mathbb{N}} X^{(k)}_j$ and
    \[
    D_j(x,y):=\hat{d}^{(k_{x,y})}_j(x,y) \text{ where } k_{x,y}:=\min\{ k \in \mathbb{N} \ : \ x,y\in X^{(k)}_j \}.
    \]
    
	Repeating this procedure for each $j$ we obtain a collection of $n$ $b$-metric spaces with relaxation constants $K_1,\dots, K_n$ respectively. Define $(\g{X},D)$ as in the \eqref{functio}, that is: 
	\[
	\g{X}:=\prod_{j=1}^n  X_j, \qquad D(\g{x},\g{y}):= F\left(D_1(x_1,y_1),\dots, D_n(x_n,y_n)\right),
	\]
    where $\g{x}=(x_1,\dots,x_n) \in \g{X}$, $\g{y}$ likewise. 
    To prove our claim, we simply need to ensure ourselves that $(\g{X},D)$ fails to be a $b$-metric space.
    
    Suppose otherwise, i.e., $(B_{K^\prime})$ inequality holds in $(\g{X},D)$ for some $K^\prime \geqslant 1$. Consider any $k_0\geqslant K^\prime$. By Remark \ref{useful:remark} if $b$-metric inequality with constant $K^\prime$ holds, then condition $(B_{k_0})$ holds as well. Consider elements $\g{x},\g{y},\g{z} \in X$ of the form
    \begin{eqnarray*}
    \g{x}&:=& ((k_0,1,1), (k_0,2,1),\dots, (k_0,n,1));\\
	\g{y}&:=& ((k_0,1,2), (k_0,2,2),\dots, (k_0,n,2));\\
	\g{z}&:=& ((k_0,1,3), (k_0,2,3),\dots, (k_0,n,3)). 
    \end{eqnarray*}
    Notice that for each $j$ the inequality
    \begin{eqnarray*}
    D_j((k_0,j,1),(k_0,j,2)) &=& a_j^{(k_0)} 
    \leqslant K_j \cdot ( b_j^{(k_0)} + c_j^{(k_0)} )\\ &\hspace{-2.5cm}=&\hspace{-1.5cm} K_j\cdot \left(
D_j((k_0,j,2),(k_0,j,3)) + D_j((k_0,j,3),(k_0,j,1)) \right) 
    \end{eqnarray*}
    holds. However, from the assumption \eqref{actuallyfork} we obtain that
    \begin{eqnarray*}
    D(\g{x},\g{y}) = F(\g{a}^{(k_0)}) 
    &>& k_0 \cdot \left( F(\g{b}^{(k_0)}) + F(\g{c}^{(k_0)}) \right)\\
    &=& k_0\cdot \left( D(\g{x},\g{z}) + D(\g{z},\g{y}) \right).
    \end{eqnarray*}
    This proves, that $D$ fails to satisfy $(B_{k_0})$ and, by Remark \ref{useful:remark}, it also fails to do so with condition $(B_{K^\prime})$. Due to $K^\prime$ being arbitrary, the semimetric space $(\g{X},D)$ fails to be a $b$-metric space at all, a contradiction. Thus $F$ satisfies a).
    The condition b) is a straightforward consequence of a).
\end{proof} 

\begin{corollary}\label{waznecorollary}
If $F\in P^n_B$ then for all $K_1,\dots,K_n\geqslant 1$ we have that $F\in P_{(B_{K_1},\dots,B_{K_n})-B_K}$ for some $K\geqslant 1$.
\end{corollary}

\begin{remark} From Lemma \ref{mainlemma} and Theorem \ref{theorem:main} it is possible to obtain an $(n)$-dimensional generalization of Theorem 4.1 from \cite{arxiv}. Among the others, it states that $F\in P^n_{MB}$ if and only if there exists $K\geqslant 1 $ such that for any $(M,\dots,M)$-triangle triplet $(\g{a},\g{b},\g{c})$, the values $\left(F(\g{a}),F(\g{b}),F(\g{c})\right)$ form a $B_K$-triangle triplet. Similar characterizations can be given for other classes like $P^n_{BS}$, $P^n_{S}$ and so forth.
\end{remark}

The conditions above are, in general, hard to verify. Therefore a need for more convenient conditions arises. We begin with the following, relatably easy to check sufficient condition.

\begin{proposition}\label{propo:3.1}
Let $n \in \N$ and $F\colon\R_+^n\to \R_+$ be a function. If $F$ is amenable, monotone and quasi-subadditive with a constant $s \geqslant 1$, then $F$ is ($n$)-$b$-metric preserving.
\end{proposition}

\begin{proof}
Let $\{(X_1, d_1), \dots , (X_n, d_n)\}$ be a collection of $b$-metric spaces with constants $K_1, \dots , K_n$ respectively. Let $\g{X} := X_1 \times \dots \times X_n$. 
Consider a function $D\colon \g{X} \times \g{X} \to \R_+$ given by $$D(\g{x}, \g{y})=F(d_1(x_1, y_1), \dots, d_n(x_n, y_n))$$ 
for each $\g{x}, \g{y} \in \g{X}$. Since $F$ is amenable, $D(\g{x}, \g{y})=0$ if and only if $d_1(x_1, y_1)= \dots =d_n(x_n, y_n)=0$, which is equivalent to (as $d_1, \dots , d_n$ are all $b$-metrics) $\g{x}=\g{y}$.
The symmetry of $D$ clearly comes from the symmetry of $d_1, \dots , d_n$.

Lastly, we shall prove that $D$ is a $b$-metric. Let $K\in \N$ be such that $K\geqslant \max\{K_1, \dots , K_n\}$. Put 
$$s':= \frac{s^{K+1}-s^2}{s-1}+s^{K}.$$ 
Let $\g{x}, \g{y}, \g{z} \in\g{X}$ and observe that from the quasi-subadditivity of $F$ it follows that for any natural $N \geqslant 2$ and any $(a_1,\dots,a_n) \in \R^n_+$, we have
\begin{eqnarray*}
F(Na_1,\dots,Na_n) &=& F((N-1)a_1+a_1,\dots,(N-1)a_n + a_n) \\
&\leqslant& sF((N-1)a_1,\dots,(N-1)a_n)+sF(a_1,\dots,a_n).
\end{eqnarray*}
If $N \geqslant 3$, then we can repeat this procedure for $N-1$, obtaining
\begin{eqnarray*}
F(Na_1,\dots,Na_n) &\leqslant& s^{2}F((N-2)a_1,\dots,(N-2)a_n)\\ & & + s^{2}F(a_1,\dots,a_n) + sF(a_1,\dots,a_n).
\end{eqnarray*}
An easy induction leads us to the following inequality
\begin{eqnarray*}
F(Na_1,\dots,Na_n) &\leqslant& s^{N-1}F(a_1,\dots,a_n) + s^{N-1}F(a_1,\dots,a_n) \\ & &+ s^{N-2}F(a_1,\dots,a_n) + \dots + sF(a_1,\dots,a_n)\\
&=& \left(\sum\limits_{i=1}^{N-1} s^i+s^{N-1}\right)F(a_1,\dots,a_n) \\
&=& \left(\frac{s^N-s}{s-1}+s^{N-1}\right)F(a_1,\dots,a_n). 
\end{eqnarray*}
Using the above inequality, the monotonocity of $F$ and the fact that $d_i(x_i,y_i)\leqslant K(d_i(x_i,z_i)+d_i(z_i,y_i)),$ for $i =1,\dots,n$, we have
{\small \begin{eqnarray*}
D(\g{x},\g{ y})&=&F(d_1(x_1, y_1), \dots, d_n(x_n, y_n))\\
&\hspace*{-3cm}\leqslant& \hspace*{-1.7cm}F(K(d_1(x_1, z_1)+d_1(z_1, y_1)), \dots , K(d_n(x_n, z_n)+d_n(z_n, y_n)))\\
&\hspace*{-3cm}\leqslant& \hspace*{-1.7cm} \left(\frac{s^K-s}{s-1}+s^{K-1}\right) F(d_1(x_1, z_1)+d_1(z_1, y_1), \dots , d_n(x_n, z_n)+d_n(z_n, y_n))\\
&\hspace*{-3cm}\leqslant& \hspace*{-1.7cm} s\cdot \left(\frac{s^K-s}{s-1}+s^{K-1}\right) \left[F(d_1(x_1, z_1), \dots , d_n(x_n, z_n))+F(d_1(z_1, y_1), \dots ,d_n(z_n, y_n))\right]
\\&\hspace*{-3cm}\leqslant& \hspace*{-1.7cm} s'(D(\g{x},\g{z})+D(\g{z}, \g{y}))
\end{eqnarray*}}
Thus, $D$ is a $b$-metric with a constant $s'$.
\end{proof}

\begin{proposition}\label{propos[a,b]}
Let $n \in \N$ and $F\colon\R_+^n\to \R_+$ be an amenable function. If there exist constants $a, c \in \mathbb{R}$ with $0<a\leqslant c$ such that
$$\forall_{ \g{x}\in\R_+^n} \, (\g{x}\neq (0, \dots, 0) \Longrightarrow F(\g{x})\in [a, c])$$
then $F$ is ($n$)-$b$-metric preserving with a constant $K = \max\{1, \frac{c}{2a}\}$.
\end{proposition}

\begin{proof}
Let $(X_1, d_1), \dots , (X_n, d_n)$ be a collection of $b$-metric spaces. Let $\g{X} := X_1 \times \dots \times X_n$. 
Consider a function $D\colon \g{X} \times \g{X} \to \R_+$ given by 
$$D(\g{x}, \g{y})=F(d_1(x_1, y_1), \dots, d_n(x_n, y_n))$$ 
for each $\g{x}, \g{y} \in \g{X}$. 
In similar fashion to the proof of Proposition \ref{propo:3.1} we can verify that $D$ fulfills conditions (S1), (S2). Let $\g{x}, \g{y}, \g{z} \in \g{X}$ and $K=\max\{1, \frac{c}{2a}\}$. Without loss of generality, we can assume that $\g{x}, \g{y}, \g{z}$ are distinct. Then
$$a\leqslant F(d_1(x_1, z_1), ... , d_n(x_n, z_n))$$
as well as
$$a\leqslant F(d_1(z_1, y_1), ... , d_n(z_n, y_n)).$$
In particular, this implies that
\begin{eqnarray*}
D(\g{x}, \g{y})&=&F(d_1(x_1, y_1), ... , d_n(x_n, y_n)) \leqslant c \leqslant K\cdot 2a \\
&\leqslant& K(F(d_1(x_1, z_1), ... , d_n(x_n, z_n))+F(d_1(z_1, y_1), ... ,d_n(z_n, y_n)))\\
&=&K(D(\g{x}, \g{z})+D(\g{z}, \g{y})),
\end{eqnarray*}
which concludes the proof.
\end{proof}
\begin{corollary} \label{PBM}
Let $n \in\N$ and $F:\R_+^n \to \R_+$. If $F$ is amenable and there exists $c>0$ such that $F(\textbf{a})\in [c,2c]$ for all $\textbf{a} \in\R_+^n \setminus \{\g{0}\}$, then $F\in P_{BM}^n$. 
\end{corollary}
In one-dimensional variant, the implication in the Corollary above can be reversed, as shown in \cite[Theorem 24]{Khemaratchakumthorn2018}. This is not the case when $n\geqslant 2$ is concerned. The proper counterexample is presented below.

\begin{example}
Let $F\colon \R^2_+\to \R_+$ be given by a formula
$$F(\g{a}) = \left\{ \begin{array}{ccc}  
0     &\mbox{if } \g{a}=(0,0)  \\
1     &\mbox{if } \g{a}=(1,0) \\
3    &\mbox{if } \g{a}=(0,1)\\
2     &\mbox{otherwise.} 
\end{array}\right. $$
We will show that $F \in P^2_{BM}$ despite not satisfying the condition from Corollary \ref{PBM}. Let $\g{a} = (a_1,a_2), \g{b} = (b_1,b_2), \g{c} = (c_1,c_2) \in \R^2_+ \setminus \{(0,0)\}$. Assume that $(F(\g{a}),F(\g{b}),F(\g{c})) \notin \Delta_M$. By Theorem \ref{theorem:main}, it suffices to show that for $i=1$ or $i=2$ we have that $(a_i,b_i,c_i) \notin \Delta_{B_K}$ for any $K \geqslant 1.$ Since $\g{a},\g{b},\g{c} \neq (0,0)$, we have that one of $\g{a},\g{b},\g{c}$ (say $\g{c}$) must be equal to $(0,1)$ and $\g{a}=\g{b}=(1,0).$ Then $F(\g{c})=3>1+1=F(\g{a})+F(\g{b})$. However $c_2> K(a_2+b_2) = 0$ for any $K \geqslant 1.$ Therefore, $(\g{a},\g{b},\g{c}) \notin \Delta_{B_K}$. Finally, $F \in P^2_{BM}.$
\end{example}
\begin{example} \label{kw}
Let $n \in \N, F\colon \R^n_+ \to \R_+$ be the sum of squares of its arguments, i.e. $$F(a_1,\dots,a_n)=\sum_{i=1}^n a_i^2.$$ 
$F$ is clearly monotone and amenable. We will show that it is quasi-subadditive.
Let $a_1,\dots,a_n,b_1,\dots,b_n \in \R_+$. Then
\begin{eqnarray*}
F(a_1+b_1,\dots,a_n+b_n)&=&\sum_{i=1}^n\left(a_i+b_i\right)^2 = \sum_{i=1}^n \left(a_i^2+2a_ib_i+b_i^2\right)\\ &\leqslant& \sum_{i=1}^n \left(2a_i^2+2b_i^2\right)=\sum_{i=1}^n 2a_i^2+\sum_{i=1}^n 2b_i^2\\&=& 2F(a_1,\dots,a_n)+2F(b_1,\dots,b_n).
\end{eqnarray*}
Hence $F$ is quasi-subbadditive and, by Proposition \ref{propo:3.1}, it is $(n)$-$b$-metric preserving.
Observe that $F$ is not $(n)$-metric preserving. Indeed, consider the triple $(1,2,3)$. Of course, $(1,2,3) \in \triangle_M$. However, $$F(3,\dots,3) = 9n > 5n = n + 4n = F(1,\dots,1)+F(2,\dots,2).$$ So, $(F(1,\dots,1),F(2,\dots,2),F(3,\dots,3)) \notin \triangle_M.$ By Theorem \ref{theorem:main}, $F$ is not $(n)$-metric preserving.
\end{example}

The above example proves that conditions from Proposition \ref{propos[a,b]} are not necessary. On the other hand, using Proposition \ref{propos[a,b]} one can easily construct an example of an ($n$)-$b$-metric preserving function indicating that the condition of monotonicity from Proposition \ref{propo:3.1} is also not necessary.

\begin{example} \label{not_monotone}
Let $n \in \N, F\colon \R^n_+ \to \R_+$ be given by the formula 
\begin{equation*}
    F(x_1, ... , x_n) = \begin{cases}
               0      & \mbox{if  } x_1 = ... = x_n = 0\\
               2               & \mbox{if  } x_1 = ... = x_n = 1\\
               1 & \text{otherwise}
           \end{cases}
\end{equation*} 
In the light of Proposition \ref{propos[a,b]}, F is ($n$)-$b$-metric preserving. Clearly, it is not monotone.
\end{example}

As the Reader can see, the monotonicity is not necessary for a function to be $(n)$-$b$-metric preserving. Theorem 3.1 states that amenability cannot be omitted and the following results show that neither can be quasi-subadditivity.

\begin{lemma}\label{lemma:jeden}
Let $n \in\N$ and $F\colon\R_+^n\to \R_+$, $F\in P_{MB}^n$. Then $F$ is quasi-subadditive.
\end{lemma}

\begin{proof}
By Remark \ref{uwaga}, there exists a constant $K \geqslant 1$ such that $F$ is $(M,\dots,M)-B_K$ preserving. 

Let $\g{a} = (a_1,\dots,a_n)$, $\g{b} = (b_1,\dots,b_n)\in\R_+^n$. It is easy to see that $(a_i, b_i, a_i+b_i) \in \triangle_M$ for $i \in \{1,\dots,n\}.$ Therefore, by Theorem \ref{theorem:main}, \[F(\g{a+b})\leqslant K(F(\g{a})+F(\g{b}))\] for every $\g{a},\g{b}\in\R_+^n$.
Hence $F$ is quasi-subadditive with a constant $K$.
\end{proof}

Despite being useful on its own, we can use this Lemma to prove the following

\begin{theorem} \label{pmb}
Let $n \in\N$ and $F:\R_+^n \to \R_+$. Then $F\in P_B^n$ if and only if $F\in P_{MB}^n$.
\end{theorem}

\begin{proof}
The necessity is obvious. Now, let $F\in P_{MB}^n$ and consider a collection of $b$-metric spaces $(X_1, d_1), \dots , (X_n, d_n)$ with relaxation constants $K_1, \dots , K_n\geqslant 1$ respectively.  Let $\g{X} := X_1 \times \dots \times X_n$. Consider a function $D\colon \g{X} \times \g{X} \to \R_+$ given by 
$$D(\g{x}, \g{y})=F(d_1(x_1, y_1), \dots, d_n(x_n, y_n))$$
for each $\g{x} = (x_1,\dots,x_n), \g{y}=(y_1,\dots,y_n) \in \g{X}$.
Condition (S1) holds for the function $D$ since $F$ is amenable. (S2) is also obvious.

Using Lemma \ref{lemma:jeden} we obtain the existence of $s\geqslant 1$ such that 
\begin{equation}\label{eqn:2.1} 
F(\g{a+b})\leqslant s(F(\g{a})+F(\g{b})) \mbox{ for all } \g{a}, \g{b}\in \R_+^n.
\end{equation}
Since $d_1, \dots , d_n$ are $b$-metrics, there exists $N\in \N$, $N\geqslant \max\{K_1, \dots, K_n\}$ such that for any $i \in \{1,\dots,n\}$ and for all $x_i,y_i,z_i \in X_i$ 
\begin{equation}\label{eqn:2.2}
d_i(x_i, y_i)\leqslant
N\left(d_i(x_i, z_i)+d_i(z_i, y_i) \right).
\end{equation}
Using Theorem \ref{theorem:main}, Lemma \ref{mainlemma} and the fact that $F\in P_{MB}^n$, we are able to find $K\geqslant 1$ such that for any $\g{a} =(a_1,\dots,a_n)$, $\g{b}=(b_1,\dots, b_n)$, $\g{c}=(c_1,\dots,c_n) \in \R_+ ^n$ satisfying $(a_i,b_i,c_i) \in \triangle_M$ for $i \in \{1,\dots,n\}$ we have $(F(\g{a}), F(\g{b}), F(\g{c})) \in \triangle_{B_K}$.

Let $M:=2K\left(\sum\limits_{i=2}^{N} s^i+s^{N}\right)$. Consider $\g{x} = (x_1,\dots,x_n)$, $\g{y}=(y_1,\dots, y_n)$, $\g{z}=(z_1,\dots,z_n)\in \g{X}$ and define $\g{a}, \g{b}, \g{c}$ as follows: \[
\g{a}:= (d_1(x_1, y_1), \dots , d_n(x_n,y_n )),\]
\[\g{b}:= (d_1(x_1, z_1), \dots , d_n(x_n, z_n)),\]
\[
\g{c}:= (d_1(y_1, z_1), \dots , d_n(y_n, z_n)).
\]
Then, by  \eqref{eqn:2.2}, we get
\begin{equation}\label{eqn:2.3}
a_i \leqslant Nb_i+Nc_i
\end{equation}
for all $i \in \{1,\dots,n\}.$
 Therefore, $(a_i,Nb_i+Nc_i,Nb_i+Nc_i) \in \triangle_M$
 for $i \in \{1,\dots,n\},$ and hence 
 $(F(\g{a}), F(N\g{b}\g{+}N\g{c}), F(N\g{b}\g{+}N\g{c}))\in \triangle_{B_K}$. Thus,
\begin{eqnarray*}\label{eqn:2.4}
D(\g{x},\g{ y})=F(\g{a})&\leqslant& K(F(N\g{b}\g{+}N\g{c})+F(N\g{b}\g{+}N\g{c}))
\\&=&2K\cdot F(N(\g{b}\g{+}\g{c})).
\end{eqnarray*}
Analogously to the reasoning in the proof of Proposition \ref{propo:3.1} we can show that for all $m\in \N$
\begin{equation*}
F(m\cdot \g{x})\leqslant \left(\sum\limits_{i=1}^{m-1} s^i+s^{m-1}\right)F(\g{x})\mbox{ for all }\g{x}\in \R_+^n.
\end{equation*}
Combining \eqref{eqn:2.1} with the two inequalities above we obtain
\begin{eqnarray*}
D(\g{x},\g{ y})&\leqslant& 2K\cdot F(N(\g{b+}\g{c}))\leqslant 2K\left(\sum\limits_{i=1}^{N-1} s^i+s^{N-1}\right)F(\g{b+}\g{c})\\
&\leqslant& 2K\left(\sum\limits_{i=2}^{N} s^i+s^{N}\right)(F(\g{b})+F(\g{c}))\\
&=&M(D(\g{x},\g{z})+D(\g{z},\g{y})),
\end{eqnarray*}
which proves that $D$ is a $b$-metric on $\g{X}$ with a constant $M$.
\end{proof}

\begin{example}
Consider $F\colon \R^2_+ \to \R_+$ given by the formula \[
F(a,b)=e^{a+b} \mbox{ for all } (a,b)\in\R^2_+.
\]
We will show that $F$ is not quasi-subadditive.
Observe that
\begin{eqnarray*}
\lim\limits_{a_1\to \infty}\lim\limits_{a_2\to \infty} \frac{F(a_1+a_2,0)}{F(a_1,0)+F(a_2,0)}&=& \lim\limits_{a_1\to \infty}\lim\limits_{a_2\to \infty} \frac{e^{a_1+a_2}}{e^{a_1}+e^{a_2}}\\ &=& \lim\limits_{a_1\to \infty}e^{a_1} = \infty.
\end{eqnarray*}
Hence for any $s \geqslant 1$ there exist $a_1,a_2 \in \R_+$ such that $$F(a_1+a_2,0) > s(F(a_1,0)+F(a_2,0)).$$ Consequently, $F$ is not quasi-subadditive and, by Lemma \ref{lemma:jeden}, it is not $(2)$-metric-$b$-metric preserving. By Theorem \ref{pmb}, $F$ is not $(2)$-$b$-metric preserving.
\end{example}


%

\begin{proposition}
Let $n \in \N$. Then $P^n_{BM}\subsetneq P^n_M \subsetneq P^n_{MB} = P^n_B.$
\end{proposition}
\begin{proof}
The inclusions $P^n_{BM}\subset P^n_M$ and  $P^n_M \subset P^n_{MB}$ follows from Lemma \ref{zal} and Remark \ref{uwaga}. 

The equality $P^n_{MB} = P^n_B$ is precisely the statement of Theorem \ref{pmb}. 

The arithmetic mean $F$ from Example \ref{srednia aryt} is $(n)$-metric preserving but it does not convert the products of $b$-metrics to metrics. Indeed, taking the  $(B_2,\dots,B_2)$-triplet $(\g{x},\g{y},\g{z})$ consisting of \begin{align*} \g{x}&:=(1,\dots, 1),& \g{y}&:=(2,\dots,2),& \g{z}&:=(6,\dots,6),
\end{align*} we obtain that $F(\g{x})=1,$ $F(\g{y})=2$ and $F(\g{z})=6.$ Since $6>2+1$, these values do not form an $M$-triplet. Therefore from Theorem \ref{theorem:main} it follows that $F\notin P^n_{BM}$.
Hence the inclusion $P^n_{BM}\subset P^n_M$ is proper as well. 

Lastly, the Example \ref{kw} shows that the inclusion $P^n_M \subset P^n_{MB}$ is also proper. 
\end{proof}

\section{Applications}

It is worth pointing out that functions which combine multiple distance functions on a product of metric-type spaces find their applications in multiple-criteria decision making (MCDM). In particular, we would like to discuss the application of property-preserving functions regarding the TOPSIS method (introduced by
Hwang and Yoon in 1981 \cite{Hwang1981} and subsequently developed afterwards, e.g. \cite{Hung2009,Jahanshahloo2006a,Jahanshahloo2006b,Opricovic2004}). TOPSIS is an acronym for the \textit{Technique for Order Preference
by Similarity to the Ideal Solution}. As the name suggests, the procedure ranks the alternatives according to two distances: the one from hypothetically ideal solution and the one from theoretically worst alternative. 

The TOPSIS algorithm starts with forming the decision matrix which is meant to represent the satisfaction coming from the choice of each alternative according to the given criterion. Then, the matrix is normalized according to some normalization procedure. During this process, the values are multiplied by the criteria weights (the process of obtaining these weights will not be discussed in this paper). Afterwards, the positive-ideal and negative-ideal solutions (often abbreviated by PIS and NIS) are constructed (usually by taking maximal and minimal possible values for each criterion). Then, the distances
of each available alternative to PIS and NIS are calculated with a proper distance measure, which stems from applying the product-wise property-preserving function (usually the metric-preserving variant, see e.g. \cite{Olson2004}).
At last, the alternatives are ranked based on the ratios of their distances from the negative-ideal solution to the sum of distances from the positive-ideal and the negative-ideal solutions. Of course, the higher the ratio, the better the alternative.

While some initial research has been done in this field, the effects of choice of function used to generate the metric for measuring distances from PIS and NIS remain unbeknownst to a large extent, especially when more general classes of semimetrics are discussed. While MCDM seems to be out of our area of expertise for a moment, we strongly underline that it remains an important field of applications for functions preserving certain classes of semimetric spaces discussed within the scope of this article.

\section{Conclusion}

Although the topic of generalizations of function preserving metric-type properties have not been investigated well in the multidimensional case, they bear strong similiraties to their one-dimensional equivalents. Investigation of those seems more important than their one-dimensional equivalents though, as their field of application seems much broader than initially thought by mathematicians, involving multiple practical aspects of computer studies such as database management, aforementioned MCDM and so forth.


\begin{thebibliography}{abcd}
\bibitem{Bakhtin1989} Bakhtin, I.A. (1989). \textit{The contraction mapping principle in almost metric spaces.}  Funct. Anal., Gos. Ped. Inst., Unianowsk, 30, 26-37

\bibitem{Bessenyei2017} Bessenyei, M., Pales, Z. (2017). \textit{A contraction principle in semimetric spaces.} J. Nonlinear Convex Anal. 18, 515-524


\bibitem{Borsik1981compo} Bors\'{i}k, J., Dobo\v{s}, J. (1981). \textit{Functions whose composition with every metric is a metric}, Math. Slovaca 31, 3-12

\bibitem{Borsik1981}  Bors\'{i}k, J., Dobo\v{s}, J. (1981). \textit{On a product of metric spaces.} Math. Slovaca 31(2), 193-205

\bibitem{Borsik1982} Bors\'{i}k, J., Dobo\v{s}, J. (1982). \textit{On metrization of the uniformity of a product of metric spaces.} Math. Slovaca, 32, 97–102

\bibitem{Borsik1988} Bors\'{i}k, J., Dobo\v{s}, J. (1988). \textit{On metric preserving functions.} Real Anal. Exchange 13, 285-293 

\bibitem{Corazza1999} Corazza, P. (1999). \textit{Introduction to Metric-Preserving Functions.} Amer. Math. Monthly 106(4), 309-323

\bibitem{Czerwik1993} Czerwik, S. (1993). \textit{ Contraction mappings in b-metric spaces.} Acta Math. Inform. Univ. Ostrav. 1, 5-11

\bibitem{Das1989} Das, P.P. (1989). \textit{Metricity preserving transforms.} Pattern Recogn. Lett. 10, 73-76

\bibitem{Dobos1994} Dobo\v{s}, J., Piotrowski, Z. (1994). \textit{Some remarks on metric preserving functions.} Real Anal. Exchange 19, 317-320

\bibitem{Dobos1996}  Dobo\v{s}, J. (1996). \textit{On modifications of the Euclidean metric on reals.} Tatra Mt. Math. Publ. 8, 51-54

\bibitem{Dobos1996Piotro} Dobo\v{s}, J., Piotrowski, Z. (1996). \textit{A note on metric preserving functions.} Int. J. Math. Math. Sci. 19, 199-200

\bibitem{Dobos1997}  Dobo\v{s}, J., Piotrowski, Z. (1997). \textit{When distance means money.} Internat. J. Math. Ed. Sci. Tech. 28, 513-518

\bibitem{Dobos1998} Dobo\v{s}, J. (1998). \textit{Metric preserving functions.}\\  \href{http://web.science.upjs.sk/jozefdobos/wp-content/uploads/2012/03/mpf1.pdf}{\tiny{\underline{http://web.science.upjs.sk/jozefdobos/wp-content/uploads/2012/03/mpf1.pdf}}}, \\Accessed on September, 1st, 2020

\bibitem{Dovgoshey2019} Dovgoshey, O. (2020). \textit{Combinatorial properties of ultrametrics and generalized ultrametrics.} Bull. Belg. Math. Soc. Simon Stevin 27(3), 379-417

\bibitem{Dovgoshey2020} Dovgoshey, O. (2020). \textit{On ultrametric-preserving functions.} Math. Slovaca 70(1), 173-182

\bibitem{Dovgoshey2009} Dovgoshey, O., Martio, O. (2009). \textit{Products of metric spaces, covering numbers, packing numbers and characterizations of ultrametric spaces.} Rev. Roumaine Math. Pures Appl. 54(5-6), 423-439

\bibitem{Frink1937} Frink, A.H. (1937). \textit{Distance functions and the metrization problem.} Bull. Amer. Math. Soc. 43, 133-142

\bibitem{Hung2009} Hung C.C., Chen L.H. (2009). \textit{A Fuzzy TOPSIS Decision Making Model with Entropy Weight under Intuitionistic Fuzzy Environment.} Proceedings of the International Multi Conference of Engineers and Computer Scientists 2009, Vol. I, IMECS 2009, March 18-20, 2009, Hong Kong

\bibitem{Hwang1981} Hwang, C.L., Yoon, K. (1981). \textit{Multiple Attribute Decision
Making.} Lecture Notes in Economics and Mathematical Systems 186. Springer-Verlag, Berlin

\bibitem{Jachymski1995} Jachymski, J.,  Matkowski, J., Swi\k{a}tkowski, T. (1995). \textit{Nonlinear contractions on semimetric spaces.} J. Appl. Anal. 1, 125-134

\bibitem{Jachymski2020} Jachymski, J., Turoboś, F. (2020). \textit{On functions preserving regular semimetrics and quasimetrics satisfying the relaxed polygonal inequality.} RACSAM Rev. R. Acad. Cienc. Exactas F\'{i}s. Nat. Ser. A Mat. 114, 159, 1-11

\bibitem{Jahanshahloo2006a} Jahanshahloo, G.R., Lofti, F.H., Izadikhah, M. (2006). \textit{An Algorithmic Method to Extend TOPSIS for Decision Making Problems with Interval Data.} Appl. Math. Comput. 175, 1375-1384

\bibitem{Jahanshahloo2006b} Jahanshahloo G.R., Lofti F.H., Izadikhah M. (2006). \textit{Extension of the TOPSIS Method for Decision-Making Problems with Fuzzy Data.} Appl. Math. Comput. 181, 1544-1551

\bibitem{Juza1956} J\r{u}za, M. (1956). \textit{A note on complete metric spaces.} Matematicko-fyzik\'{a}lny \v{c}asopis SAV 6, 143-148

\bibitem{Karapinar2018} Karap\i{}nar, E . (2018). \textit{A Short Survey on the Recent Fixed Point Results on $b$-Metric Spaces.} Constr. Math. Anal. 1(1), 15-44

\bibitem{Khemaratchakumthorn2018} Khemaratchatakumthorn, T., Pongsriiam, P. (2018). \textit{Remarks on $b$-metric and metric-preserving functions.} Math. Slovaca 68(5), 1009-1016

\bibitem{Khemaratchakumthorn2019} Khemaratchatakumthorn, T., Pongsriiam, P., Samphavat, S. (2019). \textit{Further remarks on $b$-metrics, metric-preserving functions, and other related metrics.} Int. J. Math. Comput. Sci. 14(2), 473-480

\bibitem{Khemaratchakumthorn2012}  Khemaratchatakumthorn, T., Termwuttipong, I. (2012). \textit{Metric-preserving functions, w-distances and Cauchy w-distances.} Thai J. Math. 5, 51-56 

\bibitem{KirkShahzad}
Kirk, W.A., Shahzad, N. (2014).
\textit{Fixed Point Theory in Distance Spaces}.
Springer, Cham 

\bibitem{Martin2011} Mart\'{i}n, J., Mayor, G., Valero, O. (2011). \textit{A fixed point theorem for asymmetric distances via aggregation functions.} Proceedings of the 6th International Summer School on Aggregation Operators, 217-222

\bibitem{Massanet2012} Massanet, S., Valero, O. (2012). \textit{New results on metric aggregation.} \textit{Proceedings of the 17th Spanish Conference on Fuzzy Technology and Fuzzy Logic (Estylf 2012)}, 558-563

\bibitem{Mayor2010} Mayor, G., Valero, O. (2010). \textit{Aggregation of Asymmetric Distances in Computer Science.} Inform. Sciences 180, 803–812

\bibitem{Mayor2019} Mayor, G.,  Valero, O. (2019). \textit{Metric aggregation functions revisited.} European J. of Combin. 80, 390-400

\bibitem{Olson2004} Olson, D. L. (2004). \textit{Comparison of weights in TOPSIS models.} Math. Comput. Modelling, 40(7-8), 721-727.

\bibitem{Opricovic2004} Opricovic, S., Tzeng, G. H. (2004). \textit{Compromise solution by MCDM methods: A comparative analysis of VIKOR and TOPSIS.} European J. Oper. Res. 156(2), 445-455

\bibitem{Pokorny1993} Pokorn\'{y}, I. (1993). \textit{Some remarks on metric-preserving functions.} Tatra Mt. Math. Publ. 2, 65-68

\bibitem{Pokorny1996} Pokorn\'{y}, I. (1996). \textit{Some remarks on metric preserving functions of several variables.} Tatra Mt. Math. Publ. 8, 89-92

\bibitem{Pokorny1998} Pokorn\'{y}, I. (1998). \textit{Remarks on the sum of metrics.} Tatra Mt. Math. Publ. 14, 63-65

\bibitem{Pongsriiam2014FPT} Pongsriiam, P., Termwuttipong, I. (2014). \textit{On metric-preserving functions and fixed point theorems}, Fixed Point Theory Appl. Article ID 2014:179, 1-14

\bibitem{Pongsriiam2014} Pongsriiam, P., Termwuttipong, I. (2014). \textit{Remarks on ultrametrics and metric-preserving functions.} Abstr. Appl. Anal. Article ID 163258, 1-9

\bibitem{Samphavat2020} Samphavat, S., Khemaratchatakumthorn, T., Pongsriiam, P. (2020). \textit{Remarks on b-metrics, ultrametrics, and metric-preserving functions.} Math. Slovaca 70(1), 61-70

\bibitem{Sreenivasan1947} Sreenivasan, T. K. (1947). \textit{Some properties of distance functions.} J. Indian Math. Soc. (N.S.) 11 (1947), 38-43

\bibitem{Termwuttipong2005} Termwuttipong, I., Oudkam, P. (2005). \textit{Total boundedness, completeness and uniform limits of metric-preserving functions.} Ital. J. Pure Appl. Math. 18, 187-196

\bibitem{Terpe1988} Terpe, F. (1988). \textit{Metric preserving functions of several variables.} Proc. Conf. Topology and Measure V, Greifswald, 169–174

\bibitem{arxiv} Turobo\'s, F. (2020). \textit{On characterization of functions preserving metric-type conditions via triangular and polygonal structures.} arXiv:2011.14110 [math.MG], 1-18

\bibitem{Vallin1998} Vallin, R. W. (1998). \textit{On metric preserving functions and infinite derivatives.} Acta Math. Univ. Comenian. (N.S.) 67(2), 373-376

\bibitem{Vallin2000} Vallin, R. W. (2000). \textit{Continuity and differentiability aspects of metric preserving functions.} Real Anal. Exchange 25(2), 849-868

\bibitem{Vallin2020} Vallin, R. W., Dovgoshey, O. A. (2019). \textit{P-adic metric preserving functions and their analogues.} arXiv:1912.10411 [math.MG], 1-21

\bibitem{VanAn2014} Van An, T., Van Dung, N., Kadelburg, Z., Radenovi\'c, S. (2014). \textit{Various generalizations of metric spaces and fixed point theorems.} RACSAM Rev. R. Acad. Cienc. Exactas F´ıs. Nat. Ser. A Mat. 109(1), 175–198

\bibitem{Wilson1935} Wilson, W.A. (1935). \textit{On certain types of continuous transformations of metric spaces.} Amer. J. Math. 57, 62-68
\end{thebibliography}
\end{document}